\documentclass[a4paper,12pt]{amsart}
	
\usepackage{fullpage}
\usepackage{amsfonts}
\usepackage{amssymb,amsthm,amsmath,color}
\usepackage{array}
\usepackage{mathrsfs}
\usepackage{dsfont}
\usepackage{lmodern}
\usepackage{bbm}
\usepackage{hyperref}
\usepackage{graphicx}
\usepackage[utf8]{inputenc}  
\usepackage[T1]{fontenc} 
\usepackage{stmaryrd}  
\usepackage{tikz}
\usepackage{tkz-euclide}
\usepackage[backend=biber]{biblatex}
\numberwithin{equation}{section}

\DeclareMathAlphabet{\mathbbo}{U}{bbold}{m}{n}

\newcommand{\sq}{\square}
\newcommand{\eps}{\varepsilon}
\newtheorem{theorem}{Theorem}[section]

\newtheorem{lem}[theorem]{Lemma}

\newtheorem{prop}[theorem]{Proposition}

\newcommand{\N}{\mathbb N}

\begin{document}

\title{Large $B_2[g]$ subsets of the first squares}

\author{F. Bosio}
\address{ENS Paris-Saclay, Centre Borelli, UMR 9010, 91190 Gif-sur-Yvette, France}
\email{francois.bosio@ens-paris-saclay.fr}

\author{R. Riblet}
\address{ENS Paris-Saclay, Centre Borelli, UMR 9010, 91190 Gif-sur-Yvette, France}
\email{robin.riblet@ens-paris-saclay.fr}

\author{J. Tarr}
\address{ENS Paris-Saclay, Centre Borelli, UMR 9010, 91190 Gif-sur-Yvette, France}
\email{jacques.tarr@ens-paris-saclay.fr}


\begin{abstract}
We prove that, for every fixed integer $g\geq 1$, the largest cardinality of a $B_2[g]$ subset of the first $n$ squares is at least a positive constant, depending only on $g$, times
$$
n^{\frac{2g}{2g+1}}(\log n)^{\frac{2-2^g}{2g+1}},
$$
for all sufficiently large $n$.
For $g=1$, this recovers the theorem of Lefmann and Thiele on Sidon subsets of the first squares. The proof follows their hypergraph method, but replaces the $4$-uniform hypergraph encoding two representations as a sum of two squares by a $2(g+1)$-uniform hypergraph encoding $g+1$ such representations. The main point is to verify that this higher-uniformity hypergraph has few edges and few $2$-cycles; the lower bound then follows from an independence theorem of Duke, Lefmann and R\"odl for linear uniform hypergraphs.
\end{abstract}

    \maketitle

\section{Introduction}

A Sidon set, also called a $B_2$ set, is a set $S$ with the property that all sums $a+b$, with $a,b\in S$ and $a\leqslant b$, are distinct. More generally, for $g\geqslant 1$, a set $A\subset \N$ is a $B_2[g]$ set if every integer has at most $g$ representations as a sum of two elements of $A$, counted as unordered pairs. Thus $B_2[1]$ sets are precisely Sidon sets. Sidon sets were introduced by S. Sidon in connection with trigonometric series \cite{Simon_Sidon} and have since become classical objects in additive number theory.

For a subset $E$ of an additive semigroup, we denote by $F_g(E)$ the largest cardinality of a $B_2[g]$ set contained in $E$:
$$F_g(E)=\max\{|A|: A\subseteq E,\ A\text{ is a }B_2[g]\text{ set}\}.$$
When $g=1$, we omit the subscript and simply write $F(E)$. The study of $F_g(E)$ is a classical problem in additive combinatorics. It is now well known (see \cite{HalberstamRoth} or \cite{OBryant}) that $F\left(\left\llbracket 1,n\right\rrbracket\right)\sim\sqrt{n}$. The lower bound was obtained independently by Chowla \cite{Chowla} and Erd\H{o}s \cite{ErdosTuran}, who established
$$  \liminf\limits_{\substack{n\rightarrow +\infty}}\dfrac{F\left(\left\llbracket 1,n\right\rrbracket\right)}{\sqrt{n}}\geqslant 1.$$
For the upper bound, Erd\H{o}s and Tur\'an \cite{ErdosTuran} proved that $F\left(\left\llbracket 1,n\right\rrbracket\right)<\sqrt{n}+O\left( n^{1/4}\right)$. This was later sharpened by Lindström \cite{Lindstrom_ameliore_E-T}, who proved that $F\left(\left\llbracket 1,n\right\rrbracket\right)<\sqrt{n}+n^{1/4}+1$. Recently, Carter, Hunter and O'Bryant \cite{OBryantCarterHunter} obtained
$$   F\left(\left\llbracket 1,n\right\rrbracket\right)<\sqrt{n}+0.98183n^{1/4}+O(1) ,$$
and, even more recently, Hou and Zhao \cite{HouZhao} submitted a preprint with the bound 
$$  F\left(\left\llbracket 1,n\right\rrbracket\right)<\sqrt{n}+0.94601n^{1/4}+O(1).$$
For fixed $g$, the analogous problem for intervals is also well understood at the level of the order of magnitude: one has
$$F_g(\llbracket 1,n\rrbracket)\asymp_g \sqrt n.$$
More precise estimates, and in particular the dependence on $g$, have been studied extensively in the literature. 

More generally, Sidon sets form the first non-trivial case of the theory of $B_h[g]$ sets. For fixed $h\geqslant 2$ and $g\geqslant 1$, the problem of determining the largest size of a $B_h[g]$ subset of an interval, and in particular the dependence on $g$, has attracted considerable attention; see
for instance \cite{CillerueloKissRuzsaVinuesa,CillerueloRuzsaVinuesa,JohnstonTaitTimmons}.

For arbitrary finite sets, much less is known. If $E\subset \mathbb R^N$ has cardinality $n$, Bailleul and Riblet \cite{RibletBailleul} proved that
 $$F_{g}(E)\geqslant\left(\frac{\sqrt{g}}{3\sqrt 3}+o(1)\right)\sqrt n \gtrsim 0.19\sqrt{gn}.$$
This is the strongest currently available general lower bound for arbitrary $n$-point subsets of Euclidean spaces. In the present note, we focus on the structured case
$$\sq_n=\{1^2,2^2,\ldots,n^2\},$$
the set of the first $n$ positive squares in $\N$.

The restriction to perfect powers leads to a rather different class of problems. Equal sums of like powers are closely related to classical questions of Waring type and to the conjecture of Lander, Parkin and Selfridge \cite{LanderParkinSelfridge}. In particular, this conjecture would imply that, for $k\geqslant 5$, the set of perfect $k$-th powers is itself a Sidon set, a question already raised by Erd\H{o}s and Graham. For generalized Sidon sets in
perfect powers, see also \cite{KissSandor,Vu}. The case of squares is special: the set of sums of two squares has zero density, by the Landau-Ramanujan theorem, whereas for $k\geqslant 3$ the corresponding set of sums of $k$
$k$-th powers is expected to have positive density; see
\cite{DeshouillersHennecartLandreau,Maynard,WooleyCubesII}. This makes the square case both more rigid and more accessible to sieve-theoretic upper-bound
arguments.

For $g=1$, Alon and Erd\H{o}s \cite{ErdosAlon} proved in particular that $F(\sq_n)\gg n^{2/3-\eps}$ for every $\eps>0$ and asked whether $F(\sq_n)=n^{1-o(1)}$. Lefmann and Thiele \cite{LefmannThiele} subsequently removed the factor $n^{-\eps}$ and obtained
$$       F(\sq_n)\gg n^{2/3}.$$
Their proof is based on the existence of large independent sets in a suitable $4$-uniform hypergraph.

Our purpose is to show that the same hypergraph strategy extends to the $B_2[g]$ setting. The resulting bound is the following.

\begin{theorem}\label{mainthm}
For every fixed integer $g\geqslant 1$, we have
$$       F_g(\sq_n)\gg_g n^{\frac{2g}{2g+1}}(\log n)^{\frac{2-2^g}{2g+1}}.$$
\end{theorem}

For $g=1$, the power of the logarithm is $0$, and Theorem \ref{mainthm} recovers the lower bound of Lefmann and Thiele. For larger $g$, the exponent of $n$ tends to $1$ as $g$ tends to infinity, at the cost of a negative logarithmic power.
%
Let us compare Theorem \ref{mainthm} with previous probabilistic
constructions. Cilleruelo \cite{CillerueloProbabilistic} proved that, for every positive integer $g$, there exists an infinite $B_2[g]$ sequence of squares $A=\{a_k\}$ satisfying
$$        a_k\leq k^{2+1/g}(\log k)^{\kappa_g}$$
for some constant $\kappa_g>0$ depending only on $g$. In finite terms, this
already yields a lower bound of the form
$$      F_g(\sq_n)\gg_g
        n^{\frac{2g}{2g+1}}(\log n)^{-\theta_g}$$
for some $\theta_g>0$. Thus Cilleruelo's construction gives the same main power of $n$ as Theorem \ref{mainthm}. Our contribution is to obtain an explicit and sharper logarithmic factor, namely
$$       (\log n)^{\frac{2-2^g}{2g+1}},$$
and to recover the logarithm-free Lefmann--Thiele bound when $g=1$. More recently, Croot, Mao, Pohoata, Sheffer and Yip \cite{CrootAndCo} obtained related lower bounds by random deletion, which are complementary in regimes where $g$ is allowed to grow with $n$.

The argument is as follows. We define a $2(g+1)$-uniform hypergraph on the vertex set $\llbracket 1,n\rrbracket$ whose edges correspond to $g+1$ disjoint representations of the same integer as a sum of two squares. Independent sets in this hypergraph give $B_2[g]$ subsets of $\sq_n$, up to the harmless distinction between sums of two distinct elements and diagonal sums. We then estimate the number of edges and the number of $2$-cycles in this hypergraph, and apply the independence theorem of Duke, Lefmann and R\"odl \cite{DukeLefmannRodl} for uniform hypergraphs with no $2$-cycles.

We also investigated upper bounds for $F_g(\sq_n)$. In particular, we developed a sieve-theoretic approach involving $L^4$ norms, which are directly related to additive energy. Combined with Tauberian arguments, this yields a non-trivial
upper bound. However, around the same time, Croot, Mao, Pohoata, Sheffer, and Yip \cite{CrootAndCo} obtained a stronger estimate, namely
$$F_g(\sq_n)\ll g^{1/4}n\exp\left(-c\frac{\log n}{\log\log n}\right),$$
where $c>0$ is an absolute constant.

\section{Preliminaries on $B_2[g]$ sets}

We shall use the following convention. A set $A\subset \N$ is a $B_2[g]$ set if, for every $m\in \N$, the number of unordered pairs $\{a,b\}\subseteq A$, possibly with $a=b$, such that $a+b=m$, is at most $g$.

It is technically convenient to first ignore diagonal representations. We say that $A$ is a $B_2[g]^*$ set if, for every $m\in\N$, the number of unordered pairs $\{a,b\}\subseteq A$ with $a\neq b$ and $a+b=m$ is at most $g$. For a finite set $X\subset \N$, write
$$ 
    F_g(X)=\max\{|A|:A\subseteq X,\ A\text{ is a }B_2[g]\text{ set}\}
$$ 
and define $F_g^*(X)$ similarly with $B_2[g]^*$ in place of $B_2[g]$.

\begin{lem}\label{starfull}
For every finite set $X\subset\N$ and every $g\geqslant 1$,
$$ 
        \frac12 F_g^*(X)\leqslant F_g(X)\leqslant  F_g^*(X).
$$ 
\end{lem}

\begin{proof}
The upper bound is immediate, since every $B_2[g]$ set is a $B_2[g]^*$ set. Let $A\subseteq X$ be a $B_2[g]^*$ set. We shall extract a $B_2[g]$ subset of $A$ of size at least $|A|/2$.

Define a directed relation $\to$ on $A$ as follows. If $2a$ has at least one representation $2a=b+c$ with $b,c\in A$, $b<c$, let $a\to b_a$, where $b_a$ is the smallest element of $A$ which occurs in such a representation. If no such representation exists, $a$ has no successor. Each element has at most one successor. Moreover, $a\to b$ implies $b<a$, and therefore the directed graph has no directed cycle.

For each $a\in A$, let $h(a)$ be the maximal length of a directed path starting from $a$. If $a\to b$, then $h(a)=h(b)+1$. Hence the parity classes of $h$ form a partition $A=A_0\sqcup A_1$ such that no element of $A_i$ points to another element of $A_i$. Choose $i$ with $|A_i|\geqslant |A|/2$ and put $B=A_i$.

We claim that $B$ is a $B_2[g]$ set. Since $A$ is $B_2[g]^*$, only diagonal sums may have more than $g$ representations in $B$ when diagonal representations are counted. Fix $a\in B$. If $2a$ has no representation by two distinct elements of $A$, then $\{a,a\}$ is its only representation involving a diagonal pair. If $2a$ has at least one such representation in $A$, then $a\to b_a$, and $b_a\notin B$. Thus at least one off-diagonal representation of $2a$ present in $A$ is removed when passing to $B$. Since there were at most $g$ off-diagonal representations in $A$, the total number of representations of $2a$ in $B$, including $\{a,a\}$, is at most $g$. This proves the claim.
\end{proof}

We shall repeatedly use estimates for the number of representations by two squares. Let
$$ 
        r_\sq(m)=\#\bigl\{\{a,b\}:1\leqslant a<b,
        \ a^2+b^2=m\bigr\}.
$$ 
The following estimates are standard. The last one is a special case of a theorem of Blomer and Granville \cite{BlomerGranville2006}.

\begin{lem}\label{squaremoments}
For every $\eps>0$ and every $m\geqslant 1$,
$$ 
        r_\sq(m)\ll_\eps m^\eps.
$$ 
Moreover, for every fixed integer $h\geqslant 1$,
$$ 
        \sum_{m\leqslant x} r_\sq(m)^h\ll_h x(\log x)^{2^{h-1}-1}.
$$ 
\end{lem}

\begin{proof}
The first estimate follows from the elementary bound $r_\sq(m)\ll d(m)$ and the classical estimate $d(m)\ll_\eps m^\eps$ for the divisor function. The moment estimate follows from \cite{BlomerGranville2006}. Passing from the usual ordered representation function to $r_\sq$ only changes the estimates by constants depending on $h$.
\end{proof}

\section{A hypergraph independence lemma}

We recall the hypergraph theorem that will be used. A hypergraph $\mathcal H=(V,\mathcal E)$ is $k$-uniform if all its edges have cardinality $k$. A $2$-cycle is a pair of distinct edges $\{E,F\}$ such that $|E\cap F|\geqslant 2$. We write $c_2(\mathcal H)$ for the number of $2$-cycles and $c_{2,i}(\mathcal H)$ for the number of pairs of distinct edges satisfying $|E\cap F|=i$. A hypergraph is linear if it has no $2$-cycle. We write $\Delta(\mathcal H)$ for its maximum degree and $I(\mathcal H)$ for its independence number.

We shall use the following maximum-degree consequence of the theorem of Duke, Lefmann and R\"odl. Their result is usually stated in terms of the average degree of a $k$-uniform hypergraph containing no $2$-cycles. Since a
bound on the maximum degree immediately gives the same bound on the average degree, the form below follows directly from their theorem.

\begin{theorem}[Duke-Lefmann-R\"odl]\label{DLR}
Let $k\geqslant 3$. There exist constants $\lambda_k,\tau_k>0$ such that the following holds. If $\mathcal H=(V,\mathcal E)$ is a finite linear $k$-uniform hypergraph with $|V|=N$ and
$$ 
        \Delta(\mathcal H)\leqslant t^{k-1}
$$ 
for some $t\geqslant \tau_k$, then
$$ 
        I(\mathcal H)\geqslant \lambda_k\frac{N}{t}(\log t)^{\frac1{k-1}}.
$$ 
\end{theorem}

The next proposition is a standard pruning device, tailored to the
hypergraphs considered below, allowing us to apply the preceding theorem to hypergraphs which are not initially linear and may have
large maximum degree. The idea is to pass to a random induced subhypergraph,
then delete one vertex from each remaining $2$-cycle and finally delete the
vertices whose degree is too large. The hypotheses are arranged so that these
deletions remove only a positive proportion of the randomly selected vertices.

\begin{prop}\label{pruning}
Let $k\geqslant 3$. For each $n$, let $\mathcal H_n=(V_n,\mathcal E_n)$ be a finite $k$-uniform hypergraph. Suppose that there exist sequences $p_n\in[0,1]$ and $t_n>0$ such that
\begin{align*}
        p_n|V_n|&\longrightarrow +\infty,\\
        t_n&\longrightarrow +\infty,\\
        c_{2,i}(\mathcal H_n)p_n^{2k-i-1}&=O(|V_n|)\qquad (2\leqslant i\leqslant k-1),\\
        |\mathcal E_n|\left(\frac{p_n}{t_n}\right)^{k-1}&=O(|V_n|).
\end{align*}
Then there exists a constant $C>0$, depending only on $k$ and on the implicit constants above, such that, for all sufficiently large $n$,
$$ 
        I(\mathcal H_n)\geqslant C|V_n|\frac{p_n}{t_n}(\log t_n)^{\frac1{k-1}}.
$$ 
\end{prop}

\begin{proof}
Multiplying $p_n$ by a sufficiently small positive constant, we may assume that, for $n$ large enough,
$$ 
        c_{2,i}(\mathcal H_n)p_n^{2k-i-1}\leqslant \frac{|V_n|}{16(k-2)}
        \qquad(2\leqslant i\leqslant k-1).
$$ 
Similarly, after replacing $t_n$ by a sufficiently large fixed constant
multiple of $t_n$, at the cost of changing only the final constant, we may assume
$$ 
        |\mathcal E_n|\left(\frac{p_n}{t_n}\right)^{k-1}\leqslant \frac{|V_n|}{32k}.
$$ 
The constants by which we rescale $p_n$ and $t_n$ depend only on $k$ and on the implicit constants in the two domination assumptions. Therefore the final constant obtained at the end of the proof has the same dependence.
Let $R_n\subseteq V_n$ be chosen by keeping each vertex independently with probability $p_n$, and let $\mathcal H'_n$ be the induced subhypergraph. By a Chernoff bound,
$$ 
        \mathbb P\left(|R_n|<\frac12p_n|V_n|\right)\to 0.
$$ 
Furthermore,
$$ 
        \mathbb E c_{2,i}(\mathcal H'_n)=c_{2,i}(\mathcal H_n)p_n^{2k-i}.
$$ 
Summing over $2\leqslant i\leqslant k-1$ gives
$$ 
        \mathbb E c_2(\mathcal H'_n)\leqslant \frac1{16}p_n|V_n|.
$$ 
By Markov's inequality,
$$ 
        \mathbb P\left(c_2(\mathcal H'_n)>\frac14p_n|V_n|\right)\leqslant \frac14.
$$ 
Likewise,
$$ 
        \mathbb E|\mathcal E'_n|=|\mathcal E_n|p_n^k,
$$ 
and therefore
$$ 
        \mathbb P\left(|\mathcal E'_n|>4|\mathcal E_n|p_n^k\right)\leqslant \frac14.
$$ 
Hence, for all sufficiently large $n$, there exists an induced subhypergraph $\mathcal H_n^*=(V_n^*,\mathcal E_n^*)$ such that
$$ 
    |V_n^*|\geqslant \frac12p_n|V_n|,
    \qquad c_2(\mathcal H_n^*)\leqslant \frac14p_n|V_n|,
    \qquad |\mathcal E_n^*|\leqslant 4|\mathcal E_n|p_n^k.
$$ 
The sum of degrees in $\mathcal H_n^*$ is at most $4k|\mathcal E_n|p_n^k$. Thus the number $T_n$ of vertices of degree larger than $t_n^{k-1}$ satisfies
$$ 
        T_n t_n^{k-1}\leqslant 4k|\mathcal E_n|p_n^k,
$$ 
and consequently, by the second normalization,
$$ 
        T_n\leqslant 4k|\mathcal E_n|\left(\frac{p_n}{t_n}\right)^{k-1}p_n
        \leqslant \frac18p_n|V_n|.
$$ 
Remove from $V_n^*$ one vertex from every $2$-cycle and all vertices of degree larger than $t_n^{k-1}$. The induced hypergraph $\mathcal H_n^\dagger=(V_n^\dagger,\mathcal E_n^\dagger)$ is linear, has maximum degree at most $t_n^{k-1}$, and satisfies
$$ 
        |V_n^\dagger|\geqslant \left(\frac12-\frac14-\frac18\right)p_n|V_n|=\frac18p_n|V_n|.
$$ 
Applying Theorem \ref{DLR} to $\mathcal H_n^\dagger$ gives
$$ 
        I(\mathcal H_n)\geqslant I(\mathcal H_n^\dagger)
        \geqslant \frac{\lambda_k}{8}|V_n|\frac{p_n}{t_n}(\log t_n)^{\frac1{k-1}},
$$ 
provided $n$ is large enough that $t_n\geqslant \tau_k$. This proves the proposition.
\end{proof}

\section{The hypergraph of additive collisions among squares}

Fix $g\geqslant 1$. For $n\geqslant 1$, we define a $2(g+1)$-uniform hypergraph
$$ 
        \mathcal H_{g,n}=(V_{g,n},\mathcal E_{g,n})
$$ 
by taking $V_{g,n}=\llbracket 1,n\rrbracket$ and by declaring that a set
$$ 
        E=\{a_1,b_1,\ldots,a_{g+1},b_{g+1}\}
$$ 
of cardinality $2(g+1)$ is an edge if the elements can be paired in such a way that
$$ 
        a_1^2+b_1^2=a_2^2+b_2^2=\cdots=a_{g+1}^2+b_{g+1}^2.
$$ 
Equivalently, an edge records $g+1$ disjoint representations of the same integer as a sum of two positive squares. We shall use the following simple observation: the witnessing pairing of an edge is unique. Indeed, if
$$
        x_1<\cdots <x_{2(g+1)}
$$
are the vertices of an edge, then the pairs are necessarily
$$
        \{x_i,x_{2(g+1)+1-i}\}, \qquad 1\leqslant i\leqslant g+1.
$$
For if $x_1$ were paired with some $x_j$ with $j<2(g+1)$, then $x_{2(g+1)}$ would be paired with some $x_\ell>x_1$, and we would have
$$
        x_\ell^2+x_{2(g+1)}^2>x_1^2+x_j^2,
$$
a contradiction. Removing this pair and repeating the same argument gives
the claimed pairing.

\begin{lem}\label{independentB2}
For every $S\subseteq V_{g,n}$, the set $S$ is independent in
$\mathcal H_{g,n}$ if and only if
$$      S^2:=\{s^2:s\in S\}\subseteq \sq_n$$
is a $B_2[g]^*$ set. In particular,
$$      F_g^*(\sq_n)= I(\mathcal H_{g,n}).$$
\end{lem}

\begin{proof}
If $S$ contains an edge, then the $g+1$ pairs defining this edge give $g+1$ distinct off-diagonal representations of the same integer as a sum of two elements of $S^2$, so $S^2$ is not $B_2[g]^*$.

Conversely, suppose that $S^2$ is not a $B_2[g]^*$ set. Then there exist $g+1$ distinct unordered pairs $\{a_i,b_i\}$, with $a_i,b_i\in S$ and $a_i\neq b_i$, such that
$$ 
        a_1^2+b_1^2=\cdots=a_{g+1}^2+b_{g+1}^2.
$$ 
Two distinct pairs occurring in this way cannot share exactly one element: if, for instance, $a^2+b^2=a^2+c^2$, then $b=c$. Thus the $g+1$ pairs are pairwise disjoint. Their union is therefore an edge of $\mathcal H_{g,n}$ contained in $S$, contradicting the independence of $S$.
\end{proof}

We now estimate the number of edges and $2$-cycles of $\mathcal H_{g,n}$.

\begin{lem}\label{edgecount}
For every fixed $g\geqslant 1$,
$$ 
        |\mathcal E_{g,n}|\ll_g n^2(\log n)^{2^g-1}.
$$ 
\end{lem}

\begin{proof}
For each $m\leqslant 2n^2$, there are at most $r_\sq(m)$ unordered representations of $m$ as a sum of two distinct squares. Choosing $g+1$ of them gives an upper bound for the number of edges associated with $m$. Hence, by Lemma \ref{squaremoments},
$$ 
        |\mathcal E_{g,n}|
        \leqslant \sum_{m\leqslant 2n^2}\binom{r_\sq(m)}{g+1}
        \ll_g \sum_{m\leqslant 2n^2} r_\sq(m)^{g+1}
        \ll_g n^2(\log n)^{2^g-1}.
$$ 
\end{proof}

\begin{lem}\label{cyclecount}
For every fixed $g\geqslant 1$ and every $\eps>0$,
$$ 
        c_2(\mathcal H_{g,n})=O_{g,\eps}(n^{2+\eps}).
$$ 
\end{lem}

\begin{proof}
It is convenient to prove the slightly stronger estimate
$$ 
        \widehat c_g(n):=c_2(\mathcal H_{g,n})+|\mathcal E_{g,n}|
        =O_{g,\eps}(n^{2+\eps})
$$ 
by induction on $g$.

We begin with $g=1$, which is the case treated by Lefmann and Thiele \cite{LefmannThiele}. We recall the elementary estimate needed here. By Lemma \ref{edgecount}, $|\mathcal E_{1,n}|\ll n^2\log n$. Let $E\in\mathcal E_{1,n}$ be fixed. The number of edges $E'$ with $|E\cap E'|=3$ is $O(1)$ for each choice of three vertices of $E$, since the fourth vertex is then determined by the relevant equation of sums of two squares. Hence such pairs contribute $O(|\mathcal E_{1,n}|)$.

It remains to consider edges $E'$ with $|E\cap E'|=2$. Fix two common
vertices $x,y\in E\cap E'$. By the uniqueness of the pairing of $E'$, there are two cases. If $x$ and $y$ are paired together in $E'$, then the other
pair of $E'$ is a representation of $x^2+y^2$ as a sum of two distinct
squares. Thus the number of possibilities is at most
$$ 
        \max_{m\leqslant 2n^2} r_\sq(m).
$$ 
If $x$ and $y$ are not paired together in $E'$, then, after relabelling the
two remaining vertices as $u$ and $v$, we have
$$ 
        x^2+u^2=y^2+v^2.
$$ 
Hence
$$ 
        |x^2-y^2|=|v^2-u^2|=|v-u||v+u|.
$$ 
Since $x\neq y$, the integer $|x^2-y^2|$ is non-zero, and the number of
possible pairs $(u,v)$ is bounded by the divisor function
$d(|x^2-y^2|)$. By the standard divisor bound, and by the estimate
$r_\sq(m)\ll_\eps m^\eps$, the number of edges $E'$ sharing two fixed
vertices with $E$ is $O_\eps(n^\eps)$. Since there are only $O(1)$ choices
for the two common vertices inside $E$, the number of edges adjacent to a
fixed edge through two vertices is $O_\eps(n^\eps)$. Therefore
$$ 
        c_2(\mathcal H_{1,n})\ll_\eps |\mathcal E_{1,n}|n^\eps\ll_\eps n^{2+\eps},
$$ 
after replacing $\eps$ by a smaller positive number. This proves the base case.

Assume now that $g\geqslant 2$ and that the result is known for $g-1$.
We count broad $2$-cycles, that is, pairs of edges which are allowed to be
equal. This only changes the quantity under consideration by adding
$|\mathcal E_{g,n}|$.

Let $(E,E')$ be such a broad $2$-cycle in $\mathcal H_{g,n}$, and choose two common vertices of $E$ and $E'$. Since each edge consists of $g+1\geqslant 3$ pairs, there is, in each edge, at least one pair avoiding these two common vertices. Removing such a pair from $E$ and from $E'$ gives two edges of $\mathcal H_{g-1,n}$, still sharing the two chosen vertices. Thus they form a broad $2$-cycle in $\mathcal H_{g-1,n}$.

Conversely, every broad $2$-cycle in $\mathcal H_{g,n}$ is obtained in this way by starting from some broad $2$-cycle in $\mathcal H_{g-1,n}$ and adding to each of the two edges one further representation of its associated integer as a sum of two distinct squares. If
$$ 
        R_n:=\max_{m\leqslant 2n^2} r_\sq(m),
$$ 
then each of the two edges has at most $R_n$ possible extensions. Hence
$$ 
        \widehat c_g(n)\leqslant \widehat c_{g-1}(n) R_n^2.
$$ 
By Lemma \ref{squaremoments}, for every $\eps>0$ we have
$R_n\ll_\eps n^\eps$. Combining this with the induction hypothesis, and
replacing $\eps$ by a smaller positive number if necessary, gives
$$ 
        \widehat c_g(n)=O_{g,\eps}(n^{2+\eps}).
$$ 
This completes the induction.
\end{proof}

\section{Proof of the main theorem}

We now apply Proposition \ref{pruning} to the hypergraphs $\mathcal H_{g,n}$. Fix $g\geqslant 1$ and put
$$ 
        k=2(g+1).
$$ 
Let
$$ 
        \eta=\frac{1}{4(g+1)(2g+1)}
$$ 
and define
$$ 
        p_n=n^{-\frac1{2g+1}+\eta}(\log n)^{-\frac{2^g-1}{2g+1}},
        \qquad
        t_n=n^\eta.
$$ 
Clearly $p_n|V_{g,n}|=np_n\to+\infty$ and $t_n\to+\infty$.

The edge condition follows from Lemma \ref{edgecount}. Indeed,
\begin{align*}
        |\mathcal E_{g,n}|\left(\frac{p_n}{t_n}\right)^{2g+1}
        &\ll_g n^2(\log n)^{2^g-1}
        \left(n^{-\frac1{2g+1}}(\log n)^{-\frac{2^g-1}{2g+1}}\right)^{2g+1}\\
        &\ll_g n.
\end{align*}

We next verify the cycle condition. Since $0<p_n\leqslant 1$ for $n$ large enough, and since $2k-i-1\geqslant k$ for every $2\leqslant i\leqslant k-1$, it is enough to show that
$$ 
        c_2(\mathcal H_{g,n})p_n^k=o(n).
$$ 
By Lemma \ref{cyclecount}, for every $\delta>0$,
$$ 
        c_2(\mathcal H_{g,n})p_n^{2(g+1)}
        \ll_{g,\delta} n^{2+\delta}
        n^{-\frac{2g+2}{2g+1}+(2g+2)\eta}.
$$ 
Since
$$ 
        \frac{1}{2g+1}-(2g+2)\eta=\frac{1}{2(2g+1)}>0,
$$ 
we may choose $\delta>0$ small enough to obtain
$$ 
        c_2(\mathcal H_{g,n})p_n^{2(g+1)}=o(n).
$$ 
Thus all hypotheses of Proposition \ref{pruning} are satisfied. Consequently, for some constant $C_g>0$ and all sufficiently large $n$,
\begin{align*}
        I(\mathcal H_{g,n})
        &\geqslant C_g n\frac{p_n}{t_n}(\log t_n)^{\frac1{2g+1}}\\
        &\geqslant C_g n^{\frac{2g}{2g+1}}
        (\log n)^{-\frac{2^g-1}{2g+1}}
        (\log n)^{\frac1{2g+1}}\\
        &=C_g n^{\frac{2g}{2g+1}}(\log n)^{\frac{2-2^g}{2g+1}}.
\end{align*}
By Lemma \ref{independentB2}, this gives the same lower bound for
$F_g^*(\sq_n)$. Finally, Lemma \ref{starfull} yields
$$ 
        F_g(\sq_n)\geqslant \frac12 F_g^*(\sq_n)
        \gg_g n^{\frac{2g}{2g+1}}(\log n)^{\frac{2-2^g}{2g+1}}.
$$
This proves Theorem \ref{mainthm}.


\printbibliography

\end{document}